\documentclass[10pt]{amsart}
\usepackage{amsmath}
\usepackage[usenames,dvipsnames]{color}
\usepackage{parskip}
\usepackage{amsfonts}
\usepackage{amscd}
\usepackage[centertags]{amsmath}
\usepackage{amssymb}
\usepackage{stmaryrd}
\usepackage[all,cmtip]{xy}
\usepackage[english]{babel}
\usepackage{tabularx}
\usepackage{amsxtra}
\usepackage{euscript}
\usepackage[T1]{fontenc}
\usepackage{doc, exscale, fontenc, latexsym, syntonly}
\usepackage{amsfonts}
\usepackage{amsthm}
\usepackage{graphicx}
\usepackage{mciteplus}
\usepackage{cite}

\newtheorem{theorem}{Theorem}[section]
\newtheorem{corollary}[theorem]{Corollary}

\newtheorem{proposition}[theorem]{Proposition}

\theoremstyle{definition}
\newtheorem{definition}[theorem]{Definition}

\theoremstyle{remark}

\numberwithin{figure}{section}
\numberwithin{table}{section}

\begin{document}

\title[$m-$homotopic Distances in Digital Images]{$m-$homotopic Distances in Digital Images}

\author{MEL\.{I}H \.{I}S}
\date{\today}

\address{\textsc{Melih İs,}
Ege University\\
Faculty of Sciences\\
Department of Mathematics\\
Izmir, Turkiye}
\email{melih.is@ege.edu.tr}

\author{\.{I}SMET KARACA}
\date{\today}

\address{\textsc{Ismet Karaca,}
	Ege University\\
	Faculty of Sciences\\
	Department of Mathematics\\
	Izmir, Turkiye}
\email{ismet.karaca@ege.edu.tr}



\begin{abstract}
    We define digital $m-$homotopic distance and its higher version. We also mention related notions such as $m-$category in the sense of Lusternik-Schnirelmann and $m-$complexity in topological robotics. Later, we examine the homotopy invariance or $m-$homotopy invariance property of these concepts.     
\end{abstract}

\maketitle

\section{Introduction}
\label{intro}

\quad One of the traditional methods for topological complexity calculation is homotopic distance, denoted by D$(,)$ \cite{VirgosLois:2022}. This notion has also been examined in digital images in recent years \cite{Borat:2021}. The digital version of D$(,)$ is not only used to express digital topological complexity (denoted by TC \cite{Farber:2003,Pavesic:2019}) of digital images or digital fibrations, but also generalizes the digital LS-category (denoted by cat \cite{CorneaLupOpTan:2003}) of digital images or maps \cite{BoratVergili:2018,BoratVergili:2020}. In \cite{VirgosLoisOprea:2023}, D$(,)$ is improved upon D$_{m}(,)$ in topological spaces. In other words, the lower bound D$_{m}(,)$ of D$(,)$ is introduced to interpret TC in terms of a new TC-related number TC$^{m}$ (called $m-$topological complexity). In addition, this also applies to cat and cat$_m$ (called $m-$category), namely that, cat cannot be greater than cat$_{m}$. 

\quad A generalized version called higher homotopic distance, denoted by D$(,...,)$, defines the higher topological complexity TC$_{n}$ in topological spaces\cite{Rudyak:2010,BoratVergili:2021}. These idea works even in simplicial complexes \cite{alaborciherdal:2024,MelihKaraca:2024}. Therefore, transferring this investigation to digital images is an important step for future motion planning studies in image processing. The relationship between $m$-homotopic distance and digital topology is rooted in the adaptation of continuous topological concepts to discrete spaces, such as digital images. The $m$-homotopic distance D$_{m}(,)$, a generalization of the conventional homotopy concept, offers a means of quantifying the dissimilarity between functions up to $m$-homotopy in continuous topology. The topological structures and attributes that emerge in digital spaces-such as those seen in computer graphics and image processing-are the subject of digital topology \cite{Ros:1979}. Taking into account the discrete structure of these spaces, the objective of applying the notion of $m-$homotopic distance to digital topology is to quantify the ''far apart'' between two digital objects (e.g., digital images). Since digital topology and $m$-homotopic distance are related in terms of application and adaption, through the application of $m$-homotopy to digital spaces, one can create new instruments and methods for examining and comprehending the topological characteristics of digital objects.

\quad The organization of the paper is given as follows. Section \ref{sec:1} presents important facts about digital topology. For instance, an adjacency relation used instead of a topological space is explained. Section \ref{sec:2} consists of two main parts. The first deals with the digital version of the notion of $m-$homotopic distance. Some results based on digital $m$-homotopic distance are also revealed. For example, D$_{m}(,)$ is a lower bound for D$(,)$. The main goal of the second part is to examine D$_{m}(,)-$related notions such as cat$_{m}$ and TC$^{m}$ in digital images. It is shown that D$_{m}(,)$ is less than or equal to all given D$_{m}(,)-$related notions. Section \ref{sec:3} is dedicated to improving the previous section by describing the generalized (or higher) versions of the given concepts. The digital setting of the higher $m-$homotopic distance with its very first results is introduced, and then, $n-$TC$^{m}$ is studied in this section. The last part of the section is interested in defining a digital fiber $m-$homotopy equivalence. Finally, it is proven that $n-$TC$^{m}$ is digitally fiber $m-$homotopy invariant.   

\section{Preliminaries}
\label{sec:1}

\quad In this section, we try to summarize the unique structure of digital topology through well-known and important definitions and results.

\quad An adjacency relation, rather than a topology, characterizes a \textit{digital image}, which is a pair operating on discrete structures composed of a subset $A$ of $\mathbb{Z}^{r}$ and a given adjacency $c_{p}$ over the points of $A$ \cite{Kong:1989}. The pair $(A,c_{p})$ is commonly used to represent it, and the adjacency relation $c_{p}$ has the following definition \cite{Kong:1989}: Let $a = (a_{1},\cdots,a_{r})$, $a^{'} = (a^{'}_{1},\cdots,a^{'}_{r})$ be any points of $\mathbb{Z}^{r}$ with $a \neq a^{'}$. Assume that $p$ is an integer such that $1 \leq p \leq r$. Then $a$ is called \textit{$c_{p}-$adjacent to $a^{'}$} if the following hold:
\begin{itemize}
	\item There exist the maximum number of $p$ indices $q \in \{1,\cdots,r\}$ for which $|a_{q}-a^{'}_{q}| = 1$.
	\item $a_{m} = a^{'}_{m}$ for any indices $m \in \{1,\cdots,r\}$ with $|a_{m}-a^{'}_{m}| \neq 1$.
\end{itemize}
For example, there is only one adjacency on $\mathbb{Z}$, and this adjacency is denoted by $c_{1}$, $2$, or $c_{1} = 2$. For $\mathbb{Z}^{2}$ and $\mathbb{Z}^{3}$, all possibilities are $c_{1} = 4$, $c_{2} = 8$ and $c_{1} = 6$, $c_{2} = 18$, $c_{3} = 26$, respectively. Note that the notation $a \leftrightarrow_{c_{p}} a^{'}$ is generally used to express that $a$ is $c_{p}-$adjacent to $a^{'}$. Furthermore, $a \leftrightarroweq_{c_{p}} a^{'}$ means that $a = a^{'} \ \ \vee \ \ a \leftrightarrow_{c_{p}} a^{'}$. 

\begin{definition}\cite{Boxer:2018}
	Let $(A,\kappa_{1})$ and $(A,\lambda_{1})$ be any digital images. Then $\kappa_{1} \geq_{d} \lambda_{1}$ provided that $$a \leftrightarrow_{\kappa_{1}} a^{'} \ \ \Rightarrow \ \ a \leftrightarrow_{\lambda_{1}} a^{'}$$ for any $a$, $a^{'}$ in $A$.
\end{definition}

\quad The notation $\kappa_{1} \geq_{d} \lambda_{1}$ is read as $\kappa_{1}$ \textit{dominates} $\lambda_{1}$.

\begin{proposition}\cite{Boxer:2018}
	\textbf{i)} Let $h : (A,\kappa) \rightarrow (A^{'},\lambda_{1})$ be a digitally $(\kappa,\lambda_{1})-$continuous function. Then $h : (A,\kappa) \rightarrow (A^{'},\lambda^{'}_{1})$ is digitally $(\kappa,\lambda^{'}_{1})-$continuous if $\lambda_{1} \geq_{d} \lambda^{'}_{1}$.
	
	\textbf{ii)} Let $h : (A,\kappa_{1}) \rightarrow (A^{'},\lambda)$ be a digitally $(\kappa_{1},\lambda)-$continuous function. Then $h : (A,\kappa^{'}_{1}) \rightarrow (A^{'},\lambda)$ is digitally $(\kappa^{'}_{1},\lambda)-$continuous if $\kappa^{'}_{1} \geq_{d} \kappa_{1}$.
\end{proposition}

\quad The notion of continuity is translated into digital images but expressed through the adjacency relation, independent of the concepts of topology and openness. Mathematically, a function $h : (A,\kappa_{1}) \rightarrow (A^{'},\lambda_{1})$ is called \textit{(digitally) $(\kappa_{1},\lambda_{1})-$continuous} \cite{Boxer:1999} provided that for any $a$, $a^{'} \in A$, $$a \leftrightarrow_{c_{p}} a^{'} \ \ \Rightarrow \ \ h(a) \leftrightarroweq_{c_{p}} h(a^{'}).$$ Similar to topological spaces, when one considers two digitally continuous maps $h$ and $h^{'}$, then $h \circ h^{'}$ and $h^{'} \circ h$ are also digitally continuous maps \cite{Boxer:1999}. A map $h : (A,\kappa_{1}) \rightarrow (A^{'},\lambda_{1})$ is said to be a \textit{digital $(\kappa_{1},\lambda_{1})-$isomorphism} if each of the following satisfies \cite{Boxer2:2006,Han:2005}:
\begin{itemize}
	\item $h$ is bijective.
	\item $h$ and $h^{-1}$ is digitally continuous.
\end{itemize} 

\quad To state a digital homotopy, we need to define a digital interval and an adjacency relation on the cartesian product for the domain of the homotopy. A digital interval \cite{Boxer:2006} is given by
\begin{eqnarray*}
	[a,a^{'}]_{\mathbb{Z}} = \{b \in \mathbb{Z} : a \leq b \leq a^{'}\}.
\end{eqnarray*}
Given a digital image $(A_{\iota},\kappa_{\iota})$ for any $\iota \in \{1,\cdots,r\}$, \textit{the normal product adjacency (or the strong product adjacency)} is denoted by NP$_{m}(\kappa_{1},\cdots,\kappa_{r})$ and given in \linebreak$A_{1} \times \cdots \times A_{r}$ as follows \cite{Berge:1976,Boxer:2017}: Let $m$ be any integer in $\{1,\cdots,r\}$, $(a_{1},\cdots,a_{r})$ and $(a^{'}_{1},\cdots,a^{'}_{r})$ any distinct points in $A_{1} \times \cdots \times A_{r}$. Then $$(a_{1},\cdots,a_{r}) \leftrightarrow_{\text{NP}_{m}(\kappa_{1},\cdots,\kappa_{r})} (a^{'}_{1},\cdots,a^{'}_{r}) \ \ \Leftrightarrow \ \ a_{j} \leftrightarrow_{\kappa_{j}} a^{'}_{j}$$ for $j \in \{1,\cdots,m\}$, and $a_{\iota} = a^{'}_{\iota}$ for all other indices $\iota$. For $m = 2$, a simple notation NP is used instead of NP$_{2}$.

\begin{theorem}\cite{BoxKar:2012}
	For any digital images $(A,\kappa_{1})$ and $(A^{'},\lambda_{1})$, we have that $$\text{NP}(\kappa_{1},\lambda_{1}) = c_{\kappa_{1}+\lambda_{1}}.$$
\end{theorem}

\quad Let $[0,z]_{\mathbb{Z}}$ be any digital interval with $2-$adjacency. Any two digitally continuous maps $h$, $k : (A,\kappa_{1}) \rightarrow (A^{'},\lambda_{1})$ are said to be \textit{digitally $(\kappa_{1},\lambda_{1})-$homotopic} \cite{Boxer:1999} if there is a digitally continuous map $$K : A \times [0,z]_{\mathbb{Z}} \rightarrow A^{'}$$ satisfying $K(a,0) = h(a)$ and $K(a,z) = k(a)$ for all $a \in A$ with the properties that, for any fixed $t \in [0,z]_{\mathbb{Z}}$, the map $$K_{t} : A \rightarrow A^{'}$$ defined by $K_{t}(a) = K(a,t)$ is digitally continuous for all $a$, and for any fixed $a \in A$, the map $$K_{a} : [0,z]_{\mathbb{Z}} \rightarrow A^{'}$$ defined by $K_{a}(t) = K(a,t)$ is digitally continuous for all $t$. Here $z$ is the number of steps of $K$. Alternatively, $K$ is called a \textit{digital homotopy} between $h$ and $k$ in $z$ steps. The notation $h \simeq_{\kappa_{1},\lambda_{1}} k$ is used to state that $h$ and $k$ are digitally homotopic to each other. Let $h : (A,\kappa_{1}) \rightarrow (A^{'},\lambda_{1})$ be a digitally continuous map. Then $h$ is called a \textit{digitally $(\kappa_{1},\lambda_{1})-$nullhomotopic} provided that $h \simeq_{\kappa_{1},\lambda_{1}} c$ for a constant map \linebreak$c : A \rightarrow A^{'}$ \cite{Boxer:1994,Khalimsky:1987}. A digital image $(A,\kappa_{1})$ is \textit{digitally $\kappa_{1}-$contractible} \cite{Boxer:1999} if $1_{A}$ and some constant map $c : A \rightarrow A$ are digitally homotopic to each other. Given a digitally $(\kappa_{1},\lambda_{1})-$continuous map $h : A \rightarrow A^{'}$, it is a \textit{digitally $(\kappa_{1},\lambda_{1})-$homotopy equivalence} \cite{Boxer:2005} provided that there is a digitally continuous map $h^{'} : A^{'} \rightarrow A$ for which both the compositions $h \circ h^{'}$ and $h^{'} \circ h$ are digitally homotopic to the identity maps. Given two digitally continuous functions $h : (A,\kappa_{1}) \rightarrow (A^{'},\lambda_{1})$ and $k : (B,\kappa^{'}_{1}) \rightarrow (A^{'},\lambda_{1})$, they are \textit{digitally fiber homotopy equivalent} maps provided that the following diagram commutes with the property that both the compositions $v \circ w$ and $w \circ v$ are digitally homotopic to the identity maps.
$$\xymatrix{
	A \ar[dr]_{h} \ar@<1ex>[rr]^v
	& & B \ar@<1ex>[ll]^w \ar[dl]^{k} \\
	& A^{'}. & }$$

\quad Let $\iota : (B,\delta_{1}) \rightarrow (B \times [0,z]_{\mathbb{Z}},\text{NP}(\delta_{1},2))$, $\iota(t) = (t,0)$, be an inclusion map for a digital image $(B,\delta_{1})$ and $z$ a positive integer. Then a digitally continuous map $h : (A,\kappa_{1}) \rightarrow (A^{'},\lambda_{1})$ has the \textit{digital homotopy lifting property} for $(B,\delta_{1})$ if for any digitally continuous map $k : (B,\delta_{1}) \rightarrow (A,\kappa_{1})$ and any digital homotopy $K : (B \times [0,z]_{\mathbb{Z}},\delta_{1}^{\ast}) \rightarrow (A^{'},\lambda_{1})$ with $h \circ k = K \circ \iota$, there exists a digitally $(\delta_{1}^{\ast},\kappa_{1})$-continuous map $$\widetilde{K} : (B \times [0,z]_{\mathbb{Z}},\delta_{1}^{\ast}) \rightarrow (A,\kappa_{1})$$ for which $h \circ \widetilde{K} = K$ and $\widetilde{K} \circ \iota = k$ \cite{EgeKaraca:2017}. A \textit{digital fibration} $h : (A,\kappa_{1}) \rightarrow (A^{'},\lambda_{1})$ is a digitally continuous map which admits the digital homotopy lifting property for every digital image $(B,\delta_{1}) $\cite{EgeKaraca:2017}. Let $a^{'}$ be any point of $(A^{'},\lambda_{1})$. Then $h^{-1}(a^{'})$ is called the \textit{digital fiber}.

\quad Recall that a digital image can have an infinite number of elements. $A \subset \mathbb{Z}^{r}$ is said to be a \textit{finite digital image} when it has a finite number of elements. For a positive integer $m > 0$, we mean an \textit{$m-$dimensional digital complex} by a finite digital image $(P,\delta) \subseteq \mathbb{Z}^{m}$. An $m-$homotopy is first defined in \cite{Fox:1941} for topological spaces (see page 344). We now express the digital version of this notion. 

\begin{definition}
	Let $h$, $k : (A,\kappa_{1}) \rightarrow (A^{'},\lambda_{1})$ be two digitally continuous maps and $(P,\delta)$ an $m-$dimensional digital complex. If $h \circ \phi \simeq_{\delta,\lambda_{1}} k \circ \phi$ for every digitally $(\delta,\kappa_{1})-$continuous map $\phi : P \rightarrow A$, then $h$ and $k$ are said to be \textit{digitally $m-$homotopic}.
\end{definition}

\quad Let $(A,\kappa_{1})$ be a digital image. Then a subset $B = \{b_{0},b_{1},...,b_{m}\} \subseteq A$ is a \textit{digital $\kappa_{1}-$path} \cite{Ros:1986} from the initial point $b_{0}$ to the desired point $b_{m}$ if $b_{\iota} \leftrightarroweq_{\kappa_{1}} b_{\iota+1}$ for $\iota = 0,1,\cdots,m-1$. $(A,\kappa_{1})$ is called \textit{$\kappa_{1}-$connected} \cite{Ros:1986} provided that there is at least one digital $\kappa_{1}-$path between any two points in $A$.

\quad Assume that $(A,\kappa_{1})$ and $(B,\lambda_{1})$ are any digital images. Then the \textit{function space $B^{A}$ of digital images} \cite{LupOpreaScov:2019} consists of a set
\begin{eqnarray*}
	\{\vartheta : A \rightarrow B \ | \ \vartheta \ \ \text{is a digitally continuous map} \}
\end{eqnarray*}
and $B^{A}$ has an adjacency relation $\delta$ such that for all $\vartheta$, $\theta \in B^{A}$ and $a$, $a^{'} \in A$, $$a \leftrightarroweq_{\kappa_{1}} a^{'} \ \ \Rightarrow \ \ \vartheta(a) \leftrightarroweq_{\lambda_{1}} \theta(a^{'}).$$ Spesifically, choose $A = [0,z]_{\mathbb{Z}}$. Assume that $h$ and $k$ are two digital paths in $B$. Then they are \textit{digitally $\delta-$adjacent to each other} \cite{KaracaIs:2018} provided that $$t_{1} \leftrightarroweq_{2} t_{2} \ \ \Rightarrow \ \ h(t_{1}) \leftrightarroweq_{\lambda_{1}} k(t_{2}).$$ As an example of a digital fibration, we have the digital map $\pi : A^{[0,z]_{\mathbb{Z}}} \rightarrow A \times A$ with $\pi(\vartheta) = (\vartheta(0),\vartheta(z))$ for any digital path $\vartheta \in A^{[0,z]_{\mathbb{Z}}}$ (see Theorem 3.27 of \cite{MelihKaraca:2020}).

\begin{proposition}\cite{MelihKaraca:2020}
	$(B^{C})^{A}$ is digitally isomorphic to $B^{A \times C}$.
\end{proposition}

\quad Recall that TC$(h;\kappa_{1},\lambda_{1})$ of a surjective digitally continuous function\linebreak $h : (A,\kappa_{1}) \rightarrow (A^{'},\lambda_{1})$ is defined by Definition 3.1 in \cite{MelihKaraca:2021}. Generalization of this in digital images is stated as follows by considering the topological setting \cite{aghimirebaba:2023,MelihKaraca:2022}.  

\begin{definition}
	The \textit{digital higher TC} of a digital fibration $h : (A,\kappa_{1}) \rightarrow (A^{'},\lambda_{1})$ is defined as $$\text{n-TC}(h;\kappa_{1},\lambda_{1}) = \text{D}(h \circ pr_{1},\cdots,h \circ pr_{n}; \text{NP}_{n}(\kappa_{1},\cdots,\kappa_{1}),\lambda_{1})$$ for the projection map $pr_{\iota} : (A^{n},\text{NP}_{n}(\kappa_{1},\cdots,\kappa_{1})) \rightarrow (A,\kappa_{1})$ with $\iota \in \{1,\cdots,n\}$.
\end{definition}

\section{$m-$homotopic Distance in Digital Images}
\label{sec:2}

\subsection{Digital $m-$homotopic Distance}
\label{subsec:1}

\quad This section is devoted to the digital image adaptation of our main character, the homotopic distance.

\begin{definition}\label{def1}
	Let $h$, $k : (A,\kappa_{1}) \rightarrow (A^{'},\lambda_{1})$ be any digitally continuous maps and $(P,\delta)$ an $m-$dimensional digital complex for $m > 0$. Then the digital $m-$homotopic distance D$_{m}(h,k;\kappa_{1},\lambda_{1})$ between $h$ and $k$ is the least integer $q \geq 0$ provided that $A = X_{0} \cup \cdots \cup X_{q}$ such that, for each $X_{j}$, $j = 0,\cdots,q$, any $(\delta,\kappa_{1})-$continuous map $\phi : P \rightarrow X_{j}$ admits the property that $h \circ \phi \simeq_{\delta,\lambda_{1}} k \circ \phi$.
\end{definition}

\quad To eliminate the difficulty in the notation, we generally use D$_{m}(h,k)$ instead of D$_{m}(h,k;\kappa_{1},\lambda_{1})$. Definition \ref{def1} has some quick observations. For example, the order of $h$ and $k$ is insignificant to compute D$_{m}$. Mathematically saying, $$\text{D}_{m}(h,k) = \text{D}_{m}(k,h).$$ Another result states that digital homotopic maps make the computation of D$_{m}$ trivial. The converse is also true, i.e., $$\text{D}_{m}(h,k) = 0 \ \Leftrightarrow \ h \simeq_{\kappa_{1},\lambda_{1}} k.$$ 

\begin{theorem}\label{thm3}
	Let $h$, $k : (A,\kappa_{1}) \rightarrow (A^{'},\lambda_{1})$ be $(\kappa_{1},\lambda_{1})-$continuous function. Then the digital $m$-homotopic distance is not greater than the digital homotopic distance, namely that, D$_{m}(h,k) \leq$ D$_{\kappa_{1},\lambda_{1}}(h,k)$.
\end{theorem}

\begin{proof}
	Assume that D$_{\kappa_{1},\lambda_{1}}(h,k) = q$. Then $A = X_{0} \cup \cdots \cup X_{q}$ such that $$h\big|_{X_{j}} \simeq_{\kappa_{1},\lambda_{1}} k\big|_{X_{j}}$$ for each $j = 0,\cdots,q$. Let $P$ be an $m-$dimensional digital complex such that $\phi : (P,\delta) \rightarrow (X_{j},\kappa_{1})$ is any digitally continuous map. Therefore, we obtain that $$h\big|_{X_{j}} \circ \phi \simeq_{\delta,\lambda_{1}} k\big|_{X_{j}} \circ \phi.$$ Since $h\big|_{X_{j}} \circ \phi = h \circ \phi$ and $k\big|_{X_{j}} \circ \phi = k \circ \phi$, it follows that $$h \circ \phi \simeq_{\delta,\lambda_{1}} k \circ \phi.$$ Thus, D$_{m}(h,k) \leq q$. 
\end{proof}

\begin{corollary}\label{cor1}
	Given any $(\kappa_{1},\lambda_{1})-$continuous functions $h$, $k : (A,\kappa_{1}) \rightarrow (A^{'},\lambda_{1})$ such that $A$ is finite and $\kappa_{1}-$connected, we have that D$_{m}(h,k) < \infty$.
\end{corollary}

\begin{proof}
	By Proposition 3.2 of \cite{Borat:2021}, we have that D$_{\kappa_{1},\lambda_{1}}(h,k) < \infty$. Theorem \ref{thm3} concludes that D$_{m}(h,k) < \infty$.
\end{proof}

\begin{proposition}
	Given any digitally $(\kappa_{1},\lambda_{1})-$continuous maps $h_{1}$, $h_{2}$, $k_{1}$, $k_{2} : (A,\kappa_{1}) \rightarrow (A^{'},\lambda_{1})$ such that $h_{1} \simeq_{\kappa_{1},\lambda_{1}} h_{2}$ and $k_{1} \simeq_{\kappa_{1},\lambda_{1}} k_{2}$, we have that D$_{m}(h_{1},k_{1}) =$ D$_{m}(h_{2},k_{2})$.
\end{proposition}

\begin{proof}
	Assume that D$_{m}(h_{1},k_{1}) = q$. Then we have that $A = X_{0} \cup \cdots \cup X_{q}$ such that, for each $X_{j}$, $j = 0,\cdots,q$, any $(\delta,\kappa_{1})-$continuous map $\phi : P \rightarrow X_{j}$ admits the property that $h_{1} \circ \phi \simeq_{\delta,\lambda_{1}} k_{1} \circ \phi$ for any digital $m-$dimensional complex $P$. Since $h_{1} \simeq_{\kappa_{1},\lambda_{1}} h_{2}$ and $k_{1} \simeq_{\kappa_{1},\lambda_{1}} k_{2}$, we obtain that $h_{2} \circ \phi \simeq_{\delta,\lambda_{1}} h_{1} \circ \phi$ and $k_{1} \circ \phi \simeq_{\delta,\lambda_{1}} k_{2} \circ \phi$, respectively. Therefore, we obtain that
	\begin{eqnarray*}
		h_{2} \circ \phi \simeq_{\delta,\lambda_{1}} h_{1} \circ \phi \simeq_{\delta,\lambda_{1}} k_{1} \circ \phi \simeq_{\delta,\lambda_{1}} k_{2} \circ \phi.
	\end{eqnarray*}
    It follows that D$_{m}(h_{2},k_{2}) \leq q$. Similarly, by using the same process, D$_{m}(h_{2},k_{2}) = q$ implies that D$_{m}(h_{1},k_{1}) \leq q$.
\end{proof}

\begin{proposition}\label{prop1}
	Let $h$, $k : (A,\kappa_{1}) \rightarrow (A^{'},\lambda_{1})$ be any digitally $(\kappa_{1},\lambda_{1})-$continuous map. Then
	
	\textbf{i)} D$_{m}(\alpha \circ h,\alpha \circ k) \leq$ D$_{m}(h,k)$ for any digitally $(\lambda_{1},\omega)-$continuous map \linebreak$\alpha : (A^{'},\lambda_{1}) \rightarrow (B,\omega)$. 
	
	\textbf{ii)} D$_{m}(h \circ \beta,k \circ \beta) \leq$ D$_{m}(h,k)$ for any digitally $(\omega,\kappa_{1})-$continuous map \linebreak $\beta : (B,\omega) \rightarrow (A,\kappa_{1})$.
\end{proposition}

\begin{proof}
	Assume that D$_{m}(h,k) = q$, i.e.,
	$A$ can be written as the union of $X_{0},\cdots,X_{q}$ such that, for each $X_{j}$, $j = 0,\cdots,q$, any $(\delta,\kappa_{1})-$continuous map $\phi : P \rightarrow X_{j}$ admits the property that $h \circ \phi \simeq_{\delta,\lambda_{1}} k \circ \phi$ for any digital $m-$dimensional complex $P$. 
	
	\textbf{i)} Then we obtain
	\begin{eqnarray*}
		h \circ \phi \simeq_{\delta,\lambda_{1}} k \circ \phi \ \Rightarrow \ (\alpha \circ h) \circ \phi \simeq_{\delta,\omega} (\alpha \circ k) \circ \phi,
	\end{eqnarray*}
	which shows that D$_{m}(\alpha \circ h,\alpha \circ k) \leq q$.
    
    \textbf{ii)} Let $Y_{j} = \beta^{-1}(X_{j})$ for each $j$. Then $B$ can be written as the union of $Y_{0},\cdots,Y_{q}$. Since $\phi$ is the composition $\beta \circ \phi^{'}$ for any $(\delta,\omega)-$continuous map $\phi^{'} : P \rightarrow Y_{j}$, we find that
    \begin{eqnarray*}
    	h \circ \phi \simeq_{\delta,\lambda_{1}} k \circ \phi \ \Rightarrow \ h \circ (\beta \circ \phi^{'}) \simeq_{\delta,\lambda_{1}} k \circ (\beta \circ \phi^{'}) \ \Rightarrow \ (h \circ \beta) \circ \phi^{'} \simeq_{\delta,\lambda_{1}} (k \circ \beta) \circ \phi^{'}.
    \end{eqnarray*}
    This shows that D$_{m}(h \circ \beta,k \circ \beta) \leq q$.
\end{proof}

\subsection{Related Notions}
\label{subsec:2}

\begin{definition}\label{def2}
    Let $(A,\kappa_{1})$ be a digital image and $(P,\delta)$ an $m-$dimensional digital complex for $m > 0$. Then the digital $m-$Lusternik-Schnirelmann category cat$_{m}(A,\kappa_{1})$ is the least integer $q \geq 0$ provided that $A = X_{0} \cup \cdots \cup X_{q}$ such that, for each $X_{j}$, $j = 0,\cdots,q$, any $(\delta,\kappa_{1})-$continuous map $\phi : P \rightarrow X_{j}$ admits the property that $i_{j} \circ \phi$ is digitally nullhomotopic (homotopic to a constant map $c : P \rightarrow A$) for the inclusion map $i_{j} : X_{j} \rightarrow A$.
\end{definition}

\quad In a specific case, if the inclusion map $i_{j}$ in Definition \ref{def2} is digitally homotopic to a constant map $c : P \rightarrow A$, then cat$_{m}(A,\kappa_{1})$ corresponds to cat$(A,\kappa_{1})$. This means that cat$_{m}(A,\kappa_{1}) \leq$ cat$(A,\kappa_{1})$.

\begin{theorem}\label{thm1}
	For any digitally $\kappa_{1}-$connected digital image $A$,
	\begin{eqnarray*}
		\text{cat}_{m}(A,\kappa_{1}) = \text{D}_{m}(1_{A},c).
	\end{eqnarray*}
\end{theorem}

\begin{proof}
	Let cat$_{m}(A,\kappa_{1}) = q$. Then $A = X_{0} \cup \cdots \cup X_{q}$ and, for each $X_{j}$, $j = 0,\cdots,q$, we have that $$i_{j} \circ \phi \simeq_{\delta,\kappa_{1}} c^{'}$$ for any $(\delta,\kappa_{1})-$continuous map $\phi : P \rightarrow X_{j}$, the inclusion map $i_{j} : X_{j} \rightarrow A$, and the constant map $c^{'} : P \rightarrow A$. Since $c^{'} = c \circ \phi$ for the constant map $c : A \rightarrow A$, it follows that
	\begin{eqnarray*}
		i_{j} \circ \phi \simeq_{\delta,\kappa_{1}} c^{'} \ \Rightarrow \ i_{j} \circ \phi \simeq_{\delta,\kappa_{1}} c \circ \phi \ \Rightarrow \ i_{j} \simeq_{\kappa_{1},\kappa_{1}} c, 
	\end{eqnarray*}
    which concludes that D$_{m}(1_{A},c) \leq q$. Conversely, if D$_{m}(1_{A},c) = q$, then we have that $A = X_{0} \cup \cdots \cup X_{q}$ and, for each $X_{j}$, $j = 0,\cdots,q$, one admits $$1_{A} \circ \phi \simeq_{\delta,\kappa_{1}} c \circ \phi$$ for any $(\delta,\kappa_{1})-$continuous map $\phi : P \rightarrow X_{j}$. Therefore, we obtain
    \begin{eqnarray*}
    	1_{A} \circ \phi \simeq_{\delta,\kappa_{1}} c \circ \phi \ \Rightarrow \ \phi \simeq_{\delta,\kappa_{1}} c^{'} \ \Rightarrow \ i_{j} \circ \phi \simeq_{\delta,\kappa_{1}} c^{'},
    \end{eqnarray*}
    which means that cat$_{m}(A,\kappa_{1}) \leq q$.
\end{proof}

\quad In Theorem \ref{thm1}, when we consider any digital maps $h : (A,\kappa_{1}) \rightarrow (A^{'},\lambda_{1})$ and $k : (A,\kappa_{1}) \rightarrow (A^{'},\lambda_{1})$ instead of $1_{A} : (A,\kappa_{1}) \rightarrow (A,\kappa_{1})$ and $c : (A,\kappa_{1}) \rightarrow (A,\kappa_{1})$, respectively, we have the following fact:

\begin{corollary}\label{cor2}
	D$_{m}(h,k) \leq$ cat$_{m}(A,\kappa_{1})$.
\end{corollary}

\begin{definition}\label{def3}
	Let $(A,\kappa_{1})$ be a digitally connected digital image and $(P,\delta)$ an $m-$dimensional digital complex for $m > 0$. Then the digital $m-$topological complexity TC$^{m}(A,\kappa_{1})$ is the least integer $q \geq 0$ provided that $A \times A = X_{0} \cup \cdots \cup X_{q}$ such that, for each $X_{j}$, $j = 0,\cdots,q$, any $(\delta,\text{NP}(\kappa_{1},\kappa_{1}))-$continuous map $\phi : P \rightarrow X_{j}$ admits the property that $\pi \circ s_{j} \circ \phi$ is digitally homotopic to $\phi$ for each digitally continuous map $s_{j} : X_{j} \subseteq A \times A \rightarrow A^{[0,z]_{\mathbb{Z}}}$, where $\pi : A^{[0,z]_{\mathbb{Z}}} \rightarrow A \times A$ is defined by $\pi(\vartheta) = (\vartheta(0),\vartheta(z))$.
\end{definition} 

\quad In a specific case, if $\phi : P \rightarrow X_{j}$ in Definition \ref{def3} is digitally homotopic to the identity map $1_{X_{j}} : X_{j} \rightarrow X_{j}$, then TC$^{m}(A,\kappa_{1})$ corresponds to TC$(A,\kappa_{1})$. This means that TC$^{m}(A,\kappa_{1}) \leq$ TC$(A,\kappa_{1})$.

\begin{theorem}\label{thm2}
	For any digitally $\kappa_{1}-$connected digital image $A$,
	\begin{eqnarray*}
		\text{TC}^{m}(A,\kappa_{1}) = \text{D}_{m}(pr_{1},pr_{2}),
	\end{eqnarray*}
    where $pr_{1}$, $pr_{2} : A \times A \rightarrow A$ are the projections.
\end{theorem}

\begin{proof}
	Let TC$^{m}(A,\kappa_{1}) = q$. Then $A \times A = X_{0} \cup \cdots \cup X_{q}$ and, for each $X_{j}$, $j = 0,\cdots,q$, any $(\delta,\text{NP}(\kappa_{1},\kappa_{1}))-$continuous map $\phi : P \rightarrow X_{j}$ admits the property that $$\pi \circ s_{j} \circ \phi \simeq_{\delta,\text{NP}(\kappa_{1},\kappa_{1})} \phi$$ for each digitally continuous map $s_{j} : X_{j} \subseteq A \times A \rightarrow A^{[0,z]_{\mathbb{Z}}}$. Since $pr_{1} \circ \phi$ is digitally homotopic to $pr_{2} \circ \phi$, we have that
	\begin{eqnarray*}
		\pi \circ s_{j} \circ \phi \simeq_{\delta,\text{NP}(\kappa_{1},\kappa_{1})} \phi \ &\Rightarrow& \ pr_{1} \circ \pi \circ s_{j} \circ \phi \simeq_{\delta,\kappa_{1}} pr_{1} \circ \phi\\
		&\Rightarrow& \ pr_{1}|_{X_{j}} \circ \phi \simeq_{\delta,\kappa_{1}} pr_{1} \circ \phi\\
		&\Rightarrow& \ pr_{1} \circ \phi \simeq_{\delta,\kappa_{1}} pr_{2} \circ \phi, 
	\end{eqnarray*}
	which concludes that D$_{m}(pr_{1},pr_{2}) \leq q$. Conversely, if D$_{m}(pr_{1},pr_{2}) = q$, then we have that $pr_{1} \circ \phi \simeq_{\delta,\kappa_{1}} pr_{2} \circ \phi$ implies $\pi \circ s_{j} \circ \phi \simeq_{\delta,\text{NP}(\kappa_{1},\kappa_{1})} 1_{P}$ by reversing the direction of the arrows. This concludes that TC$^{m}(A,\kappa_{1}) \leq q$.
\end{proof}

\quad In Theorem \ref{thm2}, when we consider any digital maps $h : (A,\kappa_{1}) \rightarrow (A^{'},\lambda_{1})$ and $k : (A,\kappa_{1}) \rightarrow (A^{'},\lambda_{1})$ instead of $pr_{1} : (A \times A,\text{NP}(\kappa_{1},\kappa_{1})) \rightarrow (A,\kappa_{1})$ and $pr_{2} : (A \times A,\text{NP}(\kappa_{1},\kappa_{1})) \rightarrow (A,\kappa_{1})$, respectively, we have the following fact:

\begin{corollary}\label{cor3}
	D$_{m}(h,k) \leq$ TC$^{m}(A,\kappa_{1})$.
\end{corollary} 

\begin{corollary}
	D$_{m}(h,k) \leq$ cat$_{m}(A,\kappa_{1}) \leq$ TC$^{m}(A,\kappa_{1}) \leq$ cat$_{m}(A \times A,\text{NP}(\kappa_{1},\kappa_{1}))$.
\end{corollary}

\begin{proof}
	By Corollary \ref{cor2}, it is enough to prove the last two inequalities
	\begin{eqnarray*}
		\text{cat}_{m}(A,\kappa_{1}) \leq \text{TC}^{m}(A,\kappa_{1}) \ \ \text{and} \ \ \text{TC}^{m}(A,\kappa_{1}) \leq \text{cat}_{m}(A \times A,\text{NP}(\kappa_{1},\kappa_{1})).
	\end{eqnarray*}
    Assume first that TC$^{m}(A,\kappa_{1}) = q$. We shall prove cat$_{m}(A,\kappa_{1}) \leq q$. Fix a point $a^{'} \in A$ and define $h : (A,\kappa_{1}) \rightarrow (A \times A,\text{NP}(\kappa_{1},\kappa_{1}))$ with $h(a) = (a,a^{'})$. By Theorem \ref{thm1}, Proposition \ref{prop1}, and Theorem \ref{thm2}, we find
    \begin{eqnarray*}
    	\text{D}_{m}(1_{A},c) = \text{D}_{m}(pr_{1} \circ h, pr_{2} \circ h) \leq \text{D}_{m}(pr_{1}, pr_{2}) = q.
    \end{eqnarray*}
    This shows that cat$_{m}(A,\kappa_{1}) \leq q$. On the other hand, let cat$_{m}(A \times A,\text{NP}(\kappa_{1},\kappa_{1}))$ is $q$. In Corollary \ref{cor2}, by using $A \times A$, $pr_{1}$, and $pr_{2}$ instead of $A$, $h$, and $k$, respectively, we obtain TC$^{m}(A,\kappa_{1}) \leq q$.
\end{proof}

\begin{definition}\label{def5}
	Let $h : (A,\kappa_{1}) \rightarrow (A^{'},\lambda_{1})$ be a digitally continuous map and $(P,\delta)$ an $m-$dimensional digital complex for $m > 0$. Then the digital $m-$Lusternik-Schnirelmann category cat$_{m}(h;\kappa_{1},\lambda_{1})$ of $h$ is the least integer $q \geq 0$ provided that $A = X_{0} \cup \cdots \cup X_{q}$ such that, for each $X_{j}$, $j = 0,\cdots,q$, any $(\delta,\kappa_{1})-$continuous map $\phi : P \rightarrow X_{j}$ admits the property that $h|_{X_{j}} \circ \phi$ is digitally nullhomotopic (homotopic to a constant map $c : P \rightarrow A^{'}$).
\end{definition}

\quad If $h = 1_{A}$ in Definition \ref{def5}, then cat$_{m}(h;\kappa_{1},\lambda_{1})$ corresponds to cat$_{m}(A,\kappa_{1})$.

\begin{theorem}
	For any digitally $(\kappa_{1},\lambda_{1})-$continuous digital map $h : A \rightarrow A^{'}$,
	\begin{eqnarray*}
		\text{cat}_{m}(h;\kappa_{1},\lambda_{1}) = \text{D}_{m}(h,c).
	\end{eqnarray*}
\end{theorem}

\begin{proof}
		Assume that cat$_{m}(h;\kappa_{1},\lambda_{1}) = q$. Then $A = X_{0} \cup \cdots \cup X_{q}$ and, for each $X_{j}$, $j = 0,\cdots,q$, we have that $$h|_{X_{j}} \circ \phi \simeq_{\delta,\lambda_{1}} c^{'}$$ for any $(\delta,\kappa_{1})-$continuous map $\phi : P \rightarrow X_{j}$ and the constant map $c^{'} : P \rightarrow A^{'}$. Since $c^{'} = c \circ \phi$ and $h|_{X_{j}} \circ \phi = h \circ \phi$ for the constant map $c : A \rightarrow A^{'}$, we get
	\begin{eqnarray*}
		h|_{X_{j}} \circ \phi \simeq_{\delta,\lambda_{1}} c^{'} \ \Rightarrow \ h|_{X_{j}} \circ \phi \simeq_{\delta,\lambda_{1}} c \circ \phi \ \Rightarrow \ h \circ \phi \simeq_{\delta,\lambda_{1}} c \circ \phi. 
	\end{eqnarray*}
	This shows that D$_{m}(h,c) \leq q$. Conversely, assume that D$_{m}(h,c) = q$. Then $A = X_{0} \cup \cdots \cup X_{q}$ and, for each $X_{j}$, $j = 0,\cdots,q$, $$h \circ \phi \simeq_{\delta,\lambda_{1}} c \circ \phi$$ for any $(\delta,\kappa_{1})-$continuous map $\phi : P \rightarrow X_{j}$. Thus, we get $h|_{X_{j}} \circ \phi \simeq_{\delta,\lambda_{1}} c^{'}$. Finally, cat$_{m}(h;\kappa_{1},\lambda_{1}) \leq q$.
\end{proof}

\begin{corollary}
	D$_{m}(h,k) \leq$ cat$_{m}(h;\kappa_{1},\lambda_{1})$.
\end{corollary}

\quad Note that if $h$ is digitally $(\kappa_{1},\lambda_{1})-$nullhomotopic, then cat$_{m}(h;\kappa_{1},\lambda_{1}) = 0$. The result is the same provided that $(A,\kappa_{1})$ or $(A^{'},\lambda_{1})$ is digitally contractible.

\begin{proposition}
	For any digitally $\kappa_{1}-$connected digital image $A$,
	\begin{eqnarray*}
		\text{cat}_{m}(\Delta;\kappa_{1},\text{NP}(\kappa_{1},\kappa_{1})) = \text{cat}_{m}(A,\kappa_{1}),
	\end{eqnarray*}
    where $\Delta$ is a diagonal map on $A$ defined by $\Delta(a) = (a,a)$.
\end{proposition}

\begin{proof}
	 Since cat$_{m}(\Delta;\kappa_{1},\text{NP}(\kappa_{1},\kappa_{1})) =$ D$_{m}(\Delta,c^{'})$ and cat$_{m}(A,\kappa_{1}) =$ D$_{m}(1_{A},c)$ for the constant maps $c^{'} : A \rightarrow A \times A$ and $c : A \rightarrow A$, we find
	 \begin{eqnarray*}
	 	\text{D}_{m}(\Delta,c^{'}) \geq \text{D}_{m}(pr_{1} \circ \Delta,pr_{1} \circ c^{'}) = \text{D}_{m}(1_{A},c)
	 \end{eqnarray*}
     for the projection $pr_{1} : A \times A \rightarrow A$
     and
     \begin{eqnarray*}
     	\text{D}_{m}(\Delta,c^{'}) = \text{D}_{m}(\Delta \circ 1_{A},\Delta \circ c) \leq \text{D}_{m}(1_{A},c)
     \end{eqnarray*}
     by using the equality $\Delta \circ c = c^{'}$. Finally, D$_{m}(\Delta,c^{'})$ equals D$_{m}(1_{A},c)$. 
\end{proof}

\begin{definition}\label{def6}
	Let $h : (A,\kappa_{1}) \rightarrow (A^{'},\lambda_{1})$ be a digital fibration and $(P,\delta)$ an $m-$dimensional digital complex for $m > 0$. Then the digital $m-$topological complexity TC$^{m}(h;\kappa_{1},\lambda_{1})$ of $h$ is the least possible integer $q \geq 0$ provided that $A \times A^{'}$ is the union $X_{0} \cup \cdots \cup X_{q}$ such that, for each $X_{j}$, $j = 0,\cdots,q$, any $(\delta,\text{NP}(\kappa_{1},\lambda_{1}))-$continuous map $\phi : P \rightarrow X_{j}$ admits the property that $\pi_{h} \circ s_{j} \circ \phi$ is digitally homotopic to the identity map $1_{P}$ for each digitally continuous map $s_{j} : X_{j} \subseteq A \times A^{'} \rightarrow A^{[0,z]_{\mathbb{Z}}}$, where $\pi_{h} : A^{[0,z]} \rightarrow A \times A^{'}$ is defined by $\pi_{h}(\vartheta) = (\vartheta(0),h \circ \vartheta(z))$.
\end{definition} 

\quad When one considers $\phi : P \rightarrow X_{j}$ in Definition \ref{def6} as a map digitally homotopic to the identity map $1_{X_{j}} : X_{j} \rightarrow X_{j}$, TC$^{m}(h;\kappa_{1},\lambda_{1})$ equals TC$(h;\kappa_{1},\lambda_{1})$. Thus, we obtain TC$^{m}(h;\kappa_{1},\lambda_{1}) \leq$ TC$(h;\kappa_{1},\lambda_{1})$.

\begin{theorem}\label{thm4}
	For any digital fibration $h : (A,\kappa_{1}) \rightarrow (A^{'},\lambda_{1})$,
	\begin{eqnarray*}
		\text{TC}^{m}(h;\kappa_{1},\lambda_{1}) = \text{D}_{m}(h \circ pr_{1},h \circ pr_{2}),
	\end{eqnarray*}
	where $pr_{1}$, $pr_{2} : A \times A \rightarrow A$ are the projections.
\end{theorem}

\begin{proof}
    First, we show that D$_{m}(h \circ pr_{1},h \circ pr_{2}) \leq q$ when $\text{TC}^{m}(h;\kappa_{1},\lambda_{1}) = q$. Then $A \times A^{'} = X_{0} \cup \cdots \cup X_{q}$, and any $(\delta,\text{NP}(\kappa_{1},\lambda_{1}))-$continuous map $\phi : P \rightarrow X_{j}$ admits that $$(1_{A} \times h) \circ \pi \circ s_{j} \circ \phi \simeq_{\delta,\text{NP}(\kappa_{1},\lambda_{1})} 1_{P}$$ for each $X_{j}$ with $j = 0,\cdots,q$. Therefore, by using the equalities $\pi_{h} = (1_{A} \times h) \circ \pi$ and $\pi_{1} \circ (1_{A} \times h) = pr_{1}$ for the projection $\pi_{1} : A \times A^{'} \rightarrow A$, we get
    \begin{eqnarray*}
    	\pi_{1} \circ (1_{A} \times h) \circ \pi \circ s_{j} \circ \phi \simeq_{\delta,\kappa_{1}} \pi_{1} \ \ \Rightarrow \ \ h \circ pr_{1} \circ \pi \circ s_{j} \circ \phi \simeq_{\delta,\lambda_{1}} h \circ \pi_{1}.
    \end{eqnarray*}
    Since $\pi \circ s_{j} \circ \phi$ is digitally homotopic to $\phi^{'}$, where $\phi^{'} : P \rightarrow Y_{j} \subseteq A \times A$ is any $(\delta,\text{NP}(\kappa_{1},\kappa_{1}))-$continuous map ($A \times A$ is the union of subsets $Y_{0},\cdots,Y_{q}$ and ($1_{A} \times h)(X_{j}) = Y_{j}$ for each $j$), we have
    \begin{eqnarray}\label{pp1}
    	h \circ pr_{1} \circ \phi^{'} \simeq_{\delta,\lambda_{1}} h \circ \pi_{1}.
    \end{eqnarray}
    Moreover, by using the equality $\pi_{2} \circ (1_{A} \times h) = pr_{2}$, we find
    \begin{eqnarray*}
    	\pi_{2} \circ (1_{A} \times h) \circ \pi \circ s_{j} \circ \phi \simeq_{\delta,\lambda_{1}} \pi_{2}
   \end{eqnarray*} 
   for the projection $\pi_{2} : A \times A^{'} \rightarrow A^{'}$. It follows that
   \begin{eqnarray}\label{pp2}
   	h \circ pr_{2} \circ \phi^{'} \simeq_{\delta,\lambda_{1}} \pi_{2}.
    \end{eqnarray}
    The next step is to present $h \circ \pi_{1} \simeq_{\text{NP}(\kappa_{1},\lambda_{1}),\lambda_{1}} \pi_{2}$ by considering (\ref{pp1}) and (\ref{pp2}). For each $j$, the digitally continuous map $s_{j} : X_{j} \rightarrow A^{[0,z]_{\mathbb{Z}}}$ can be rewritten as $t_{j} : X_{j} \times [0,z]_{\mathbb{Z}} \rightarrow A$. Then $h \circ t_{j} : X_{j} \times [0,z]_{\mathbb{Z}} \rightarrow A^{'}$ is a digital homotopy between $h \circ \pi_{1}$ and $\pi_{2}$. Thus, D$_{m}(h \circ pr_{1},h \circ pr_{2}) \leq q$. The converse is also true, namely that, $\text{TC}^{m}(h;\kappa_{1},\lambda_{1}) \leq q$ when D$_{m}(h \circ pr_{1},h \circ pr_{2}) = q$. Indeed, (\ref{pp1}) and (\ref{pp2}) hold again, and in addition, if $h \circ \pi_{1}$ is digitally homotopic to $\pi_{2}$, then $\pi_{h}$ admits a digitally continuous map $s_{j}$ that satisfies $\pi_{h} \circ s_{j} \circ \phi \simeq_{\delta,\text{NP}(\kappa_{1},\lambda_{1})} 1_{P}$: Assume that $h \circ \pi_{1} \simeq_{\text{NP}(\kappa_{1},\lambda_{1}),\lambda_{1}} \pi_{2}$, i.e., there is a digital homotopy $K : X_{j} \times [0,z]_{\mathbb{Z}} \rightarrow A^{'}$ with $K(x,0) = h \circ \pi_{1}$ and $H(x,z) = \pi_{2}$. Then the fibration $h$ admits a lifting $K^{'} : X_{j} \times [0,z]_{\mathbb{Z}} \rightarrow A$ with $K^{'}(x,0) = \pi_{1}$. Thus, TC$^{m}(h;\kappa_{1},\lambda_{1}) \leq q$.   
\end{proof}

\quad Above theorem has a quick result:

\begin{corollary}\label{cor4}
	D$_{m}(h,k) \leq$ TC$^{m}(h;\kappa_{1},\lambda_{1})$.
\end{corollary} 

\section{Higher Version of $m-$homotopic Distance}
\label{sec:3}

\subsection{Digital Higher $m-$homotopic Distance}
\label{subsec:3}

\begin{definition}\label{def4}
	Let $h_{1},h_{2},\cdots,h_{n} : (A,\kappa_{1}) \rightarrow (A^{'},\lambda_{1})$ be any digitally continuous maps and $(P,\delta)$ an $m-$dimensional digital complex for $m > 0$. Then the digital higher $m-$homotopic distance D$_{m}(h_{1},h_{2},\cdots,h_{n};\kappa_{1},\lambda_{1})$ between $h_{1},h_{2},\cdots,h_{n}$ is the least integer $q \geq 0$ provided that $A = X_{0} \cup \cdots \cup X_{q}$ such that, for each $X_{j}$, $j = 0,\cdots,q$, any $(\delta,\kappa_{1})-$continuous map $\phi : P \rightarrow X_{j}$ admits the property that $$h_{1} \circ \phi \simeq_{\delta,\lambda_{1}} h_{2} \circ \phi \simeq_{\delta,\lambda_{1}} \cdots \simeq_{\delta,\lambda_{1}} h_{n} \circ \phi.$$
\end{definition}

\quad Instead of using D$_{m}(h_{1},h_{2}, \cdots,h_{n};\kappa_{1},\lambda_{1})$, we prefer D$_{m}(h_{1},\cdots,h_{n})$ to avoid the notation's complexity. Definition \ref{def4} has important facts as follows:
\begin{itemize}
	\item D$_{m}(h_{1},\cdots,h_{i},\cdots,h_{j},\cdots,h_{n}) =$ D$_{m}(h_{1},\cdots,h_{j},\cdots,h_{i},\cdots,h_{n})$ for any $i$, $j \in \{1,\cdots,n\}$.
	\item $h_{i}$ is digitally $(\kappa_{1},\lambda_{1})-$homotopic to $h_{i+1}$ for any $i \in \{1,\cdots,n-1\}$ if and only if D$_{m}(h_{1},\cdots,h_{n}) = 0$.
	\item D$_{m}(h_{1},\cdots,h_{i}) \leq$ D$_{m}(h_{1},\cdots,h_{i},h_{i+1})$ for any $i \in \{2,\cdots,n-1\}$.
	\item D$_{m}(h_{1},\cdots,h_{n}) =$ D$_{m}(k_{1},\cdots,k_{n})$ provided that $h_{i}$ is homotopic to $k_{i}$ in the digital sense for any $i \in \{1,\cdots,n\}$.
\end{itemize}

\begin{proposition}
	Given digitally continuous maps $h_{1},\cdots,h_{n} : (A,\kappa_{1}) \rightarrow (A^{'},\lambda_{1})$, $\alpha_{1},\cdots,\alpha_{n} : (A^{'},\lambda_{1}) \rightarrow (B,\omega)$, and $\beta_{1},\cdots,\beta_{n} : (B,\omega) \rightarrow (A,\kappa_{1})$, we have that
	
	\textbf{i)} D$_{m}(\alpha_{1} \circ h_{1},\cdots,\alpha_{n} \circ h_{n}) \leq$ D$_{m}(h_{1},\cdots,h_{n})$ provided that $\alpha_{i} \simeq_{\lambda_{1},\omega} \alpha_{i+1}$ for any $i \in \{1,\cdots,n-1\}$. 
	
	\textbf{ii)} D$_{m}(h_{1} \circ \beta_{1},\cdots,h_{n} \circ \beta_{n}) \leq$ D$_{m}(h_{1},\cdots,h_{n})$ provided that $\beta_{i} \simeq_{\omega,\kappa_{1}} \beta_{i+1}$ for any $i \in \{1,\cdots,n-1\}$.
\end{proposition}

\begin{proof}
	Assume that D$_{m}(h_{1},\cdots,h_{n}) = q$. Then $A = X_{0} \cup \cdots \cup X_{q}$, and for each $X_{j}$, $j = 0,\cdots,q$, any $(\delta,\kappa_{1})-$continuous map $\phi : P \rightarrow X_{j}$ admits the property that $$h_{1} \circ \phi \simeq_{\delta,\lambda_{1}} \cdots \simeq_{\delta,\lambda_{1}} h_{n} \circ \phi.$$
	
	\textbf{i)} We have that $$\alpha_{1} \circ (h_{1} \circ \phi) \simeq_{\delta,\omega} \cdots \simeq_{\delta,\omega} \alpha_{n} \circ (h_{n} \circ \phi) \ \ \Rightarrow \ \ (\alpha_{1} \circ h_{1}) \circ \phi \simeq_{\delta,\omega} \cdots \simeq_{\delta,\omega} (\alpha_{n} \circ h_{n}) \circ \phi,$$
	which means that D$_{m}(\alpha_{1} \circ h_{1},\cdots,\alpha_{n} \circ h_{n}) \leq q$.
	
	\textbf{ii)} If we define $Y_{j} = \beta_{i}^{-1}(X_{j})$ for each $i \in \{1,\cdots,n\}$, then $B = Y_{0} \cup \cdots \cup Y_{q}$. For each $Y_{j}$, $j = 0,\cdots,q$, any $(\delta,\omega)-$continuous map $\phi^{'} : P \rightarrow Y_{j}$ satisfies $\phi = \beta_{i} \circ \phi^{'}$. It follows that
	\begin{eqnarray*}
		h_{1} \circ (\beta_{1} \circ \phi^{'}) \simeq_{\delta,\lambda_{1}} h_{n} \circ (\beta_{n} \circ \phi^{'}) \ \ \Rightarrow \ \ (h_{1} \circ \beta_{1}) \circ \phi^{'} \simeq_{\delta,\lambda_{1}} (h_{n} \circ \beta_{n}) \circ \phi^{'}.
	\end{eqnarray*}
    Thus, we conclude that D$_{m}(h_{1} \circ \beta_{1},\cdots,h_{n} \circ \beta_{n}) \leq q$.
\end{proof}

\quad Theorem \ref{thm3} can be easily generalized as follows:

\begin{theorem}
	D$_{m}(h_{1},\cdots,h_{n}) \leq$ D$_{\kappa_{1},\lambda_{1}}(h_{1},\cdots,h_{n})$.
\end{theorem}

\begin{proof}
	The proof follows the same way as the proof of Theorem \ref{thm3}. Indeed, for all $l \in \{1,\cdots,n\}$, 
	\begin{eqnarray*}
		h_{l}|_{X_{j}} \circ \phi = h_{l} \circ \phi \ \ \Rightarrow \ \ h_{1} \circ \phi \simeq_{\delta,\lambda_{1}} \cdots \simeq_{\delta,\lambda_{1}} h_{n} \circ \phi.
	\end{eqnarray*}
\end{proof}

\quad Let us examine the effect of changes in the adjacency relation on D$_{m}(,\cdots,)$.

\begin{proposition}
	Assume that $h_{i}, k_{i} : (A,\kappa) \rightarrow (A^{'},\lambda_{i})$ is digitally continuous maps for every $i \in \{1,\cdots,n\}$. Then 
	\begin{eqnarray*}
		\text{D}_{m}(h_{1},\cdots,h_{n};\kappa,\lambda_{i}) \leq \text{D}_{m}(k_{1},\cdots,k_{n};\kappa,\lambda_{n})
	\end{eqnarray*}
    provided that $\lambda_{n} \geq_{d} \lambda_{n-1} \geq_{d} \cdots \geq_{d} \lambda_{1}$ and $h_{i} \simeq_{\kappa,\lambda_{i}} k_{i}$ for each $i$.
\end{proposition}

\begin{proof}
	Let D$_{m}(k_{1},\cdots,k_{n};\kappa,\lambda_{n}) = q$. Then $A = X_{0} \cup \cdots \cup X_{q}$ and, for each $X_{j}$, $j = 0,\cdots,q$, any $(\delta,\kappa)-$continuous map $\phi : P \rightarrow X_{j}$ admits the property that $k_{l} \circ \phi \simeq_{\delta,\lambda_{n}} k_{l+1} \circ \phi$ for every $l \in \{1,\cdots,n-1\}$. Since $\lambda_{n} \geq_{d} \lambda_{i}$, we obtain that $k_{l} \circ \phi \simeq_{\delta,\lambda_{i}} k_{l+1} \circ \phi$ for every $l$. Also, $h_{i} \simeq_{\kappa,\lambda_{i}} k_{i}$ implies that $h_{i} \circ \phi \simeq_{\delta,\lambda_{i}} k_{i} \circ \phi$ for all $i$. By using these facts, we get
	\begin{eqnarray*}
		h_{l} \circ \phi \simeq_{\delta,\lambda_{i}} k_{l} \circ \phi \simeq_{\delta,\lambda_{i}} k_{l+1} \circ \phi \simeq_{\delta,\lambda_{i}} h_{l+1} \circ \phi. 
	\end{eqnarray*}
    Finally, D$_{m}(h_{1},\cdots,h_{n};\kappa,\lambda_{i}) \leq q$.
\end{proof}

\begin{proposition}
	Assume that $h_{i}, k_{i} : (A,\kappa_{i}) \rightarrow (A^{'},\lambda)$ is digitally continuous maps for every $i \in \{1,\cdots,n\}$. Then 
	\begin{eqnarray*}
		\text{D}_{m}(h_{1},\cdots,h_{n};\kappa_{i},\lambda) \geq \text{D}_{m}(k_{1},\cdots,k_{n};\kappa_{n},\lambda)
	\end{eqnarray*}
	provided that $\kappa_{n} \geq_{d} \kappa_{n-1} \geq_{d} \cdots \geq_{d} \kappa_{1}$ and $h_{i} \simeq_{\kappa_{i},\lambda} k_{i}$ for each $i$.
\end{proposition}

\begin{proof}
	Let D$_{m}(h_{1},\cdots,h_{n};\kappa_{i},\lambda) = q$. Then $A = X_{0} \cup \cdots \cup X_{q}$ and, for each $X_{j}$, $j = 0,\cdots,q$, any $(\delta,\kappa)-$continuous map $\phi : P \rightarrow X_{j}$ admits the property that 
	\begin{eqnarray*}
		(P,\delta) \stackrel{\phi}{\rightarrow} (X_{j},\kappa_{i}) \stackrel{h_{l}}{\rightarrow} (A^{'},\lambda) \ \ \text{is digitally homotopic to} \ \ (P,\delta) \stackrel{\phi}{\rightarrow} (X_{j},\kappa_{i}) \stackrel{h_{l+1}}{\rightarrow} (A^{'},\lambda)
	\end{eqnarray*}
    for every $l \in \{1,\cdots,n-1\}$. Since $\kappa_{n} \geq_{d} \kappa_{i}$, we obtain that
    \begin{eqnarray*}
    	(P,\delta) \stackrel{\phi}{\rightarrow} (X_{j},\kappa_{n}) \stackrel{h_{l}}{\rightarrow} (A^{'},\lambda) \ \ \text{is digitally homotopic to} \ \ (P,\delta) \stackrel{\phi}{\rightarrow} (X_{j},\kappa_{n}) \stackrel{h_{l+1}}{\rightarrow} (A^{'},\lambda).
    \end{eqnarray*}
    Also, $h_{i} \simeq_{\kappa_{i},\lambda} k_{i}$ implies that 
    \begin{eqnarray*}
    	(P,\delta) \stackrel{\phi}{\rightarrow} (X_{j},\kappa_{n}) \stackrel{h_{l}}{\rightarrow} (A^{'},\lambda) \ \ \text{is digitally homotopic to} \ \ (P,\delta) \stackrel{\phi}{\rightarrow} (X_{j},\kappa_{n}) \stackrel{k_{l}}{\rightarrow} (A^{'},\lambda)
    \end{eqnarray*}
    by $\kappa_{n} \geq_{d} \kappa_{i}$ for all $i$.
    The composition of these facts gives us
    \begin{eqnarray*}
    	(P,\delta) \stackrel{\phi}{\rightarrow} (X_{j},\kappa_{n}) \stackrel{k_{l}}{\rightarrow} (A^{'},\lambda) \ \ \text{is digitally homotopic to} \ \ (P,\delta) \stackrel{\phi}{\rightarrow} (X_{j},\kappa_{n}) \stackrel{k_{l+1}}{\rightarrow} (A^{'},\lambda).
    \end{eqnarray*}
    Consequently, \text{D}$_{m}(k_{1},\cdots,k_{n};\kappa_{n},\lambda) \leq q$.
\end{proof}

\subsection{Related Notions To The Higher Version}
\label{subsec:4}

\begin{definition}
	Let $n>1$. For any digitally $\kappa_{1}-$connected digital image $A$,
	\begin{eqnarray*}
		\text{n-TC}^{m}(A,\kappa_{1}) = \text{D}_{m}(pr_{1},\cdots,pr_{n}),
	\end{eqnarray*}
	where $pr_{1},\cdots,pr_{n} : A^{n} \rightarrow A$ are the projections.
\end{definition}

\quad If $n=1$, then it is always assumed that $1-$TC$^{m}(A,\kappa_{1}) = 1$. It is easy to say that $2-$TC$^{m}(A,\kappa_{1}) =$ TC$^{m}(A,\kappa_{1})$. Moreover, $(n+1)-$TC$^{m}(A,\kappa_{1}) \geq n-$TC$^{m}(A,\kappa_{1})$.

\begin{proposition}
	For any digitally $(\kappa_{1},\lambda_{1})-$continuous maps $h_{1},\cdots,h_{n} : A \rightarrow B$ with the digitally connected domain and range,
	\begin{eqnarray*}
		\text{D}_{m}(h_{1},\cdots,h_{n}) \leq \text{n-TC}^{m}(B,\lambda_{1}).
	\end{eqnarray*}
\end{proposition}

\begin{proof}
	Consider a diagonal map $\Delta : (A,\kappa_{1}) \rightarrow (A^{n},\text{NP}(\kappa_{1},\cdots,\kappa_{1}))$ defined as $\Delta(a) = (a,\cdots,a)$, a map $\varphi := (h_{1},\cdots,h_{n}) : (A,\kappa_{1}) \rightarrow (B^{n},\text{NP}(\lambda_{1},\cdots,\lambda_{1}))$ with $\varphi(a) = (h_{1}(a),\cdots,h_{n}(a))$, and a projection $pr^{'}_{i} : (B^{n},\text{NP}(\lambda_{1},\cdots,\lambda_{1})) \rightarrow (B,\lambda_{1})$ for each $i \in \{1,\cdots,n\}$. Note that $pr^{'}_{i} \circ \varphi \circ \Delta = h_{i}$. Therefore, we get
	\begin{eqnarray*}
		\text{D}_{m}(pr^{'}_{1},\cdots,pr^{'}_{n}) \geq \text{D}_{m}(pr^{'}_{1} \circ \varphi \circ \Delta,\cdots,pr^{'}_{n} \circ \varphi \circ \Delta) = \text{D}_{m}(h_{1},\cdots,h_{n}).
	\end{eqnarray*}
\end{proof}

\begin{proposition}
	For any digitally $(\kappa_{1},\lambda_{1})-$continuous maps $h_{1},\cdots,h_{n} : A \rightarrow B$ with the digitally connected domain and range,
	\begin{eqnarray*}
		\text{D}_{m}(h_{1},\cdots,h_{n}) \leq \text{cat}_{m}(A,\kappa_{1}).
	\end{eqnarray*}
\end{proposition}

\begin{proof}
	We shall show that D$_{m}(h_{1},\cdots,h_{n}) \leq q$ when D$_{m}(1_{A},c) = q$ by Theorem \ref{thm1}. Then $A = X_{0} \cup \cdots \cup X_{q}$ and, for each $X_{j}$, $j = 0,\cdots,q$, the property $$1_{A} \circ \phi \simeq_{\delta,\kappa_{1}} c \circ \phi$$ holds for any $(\delta,\kappa_{1})-$continuous map $\phi : P \rightarrow X_{j}$. This means that $\phi$ is digitally nullhomotopic. Thus, for every $i \in \{1,\cdots,n\}$, $h_{i} \circ \phi$ is digitally nullhomotopic, i.e.,
	\begin{eqnarray*}
		h_{1} \circ \phi \simeq_{\delta,\lambda_{1}} \cdots \simeq_{\delta,\lambda_{1}} h_{n} \circ \phi \ : \ \text{constant}
	\end{eqnarray*} 
    Consequently, D$_{m}(h_{1},\cdots,h_{n}) \leq q$. 
\end{proof}

\quad Theorem \ref{thm4} is improved as follows:

\begin{corollary}
	For any digital fibration $h : (A,\kappa_{1}) \rightarrow (A^{'},\lambda_{1})$,
	\begin{eqnarray*}
		\text{n-TC}^{m}(h;\kappa_{1},\lambda_{1}) = \text{D}_{m}(h \circ pr_{1},h \circ pr_{2},\cdots h \circ pr_{n}),
	\end{eqnarray*}
	where $pr_{1},pr_{2},\cdots,pr_{n} : A \times A \rightarrow A$ are the projections.
\end{corollary}

\begin{corollary}
	$n-$TC$^{m}(h;\kappa_{1},\lambda_{1}) \leq$ $n-$TC$(h,\kappa_{1},\lambda_{1})$.
\end{corollary}

\begin{proof}
	The result comes from Theorem \ref{thm3}.
\end{proof}

\begin{theorem}
	Assume that the following diagram is commutative up to digital homotopy for each $i \in \{1,\cdots,n\}$, where $\omega_{1} : B \simeq_{\lambda_{1},\lambda^{'}_{1}} B^{'}$ and $\omega_{2} : A^{'} \simeq_{\kappa^{'}_{1},\kappa_{1}} A$ are digital homotopy equivalences.
	$$\xymatrix{
		A \ar[r]^{h_{i}} &
		B \ar[d]^{\omega_{1}} \\
		A^{'} \ar[u]^{\omega_{2}} \ar[r]_{k_{i}} & B^{'},}$$
	i.e., the property $\omega_{1} \circ h_{i} \circ \omega_{2} \simeq_{\kappa^{'}_{1},\lambda^{'}_{1}} k_{i}$ holds for each $i \in \{1,\cdots,n\}$. Then D$_{m}(h_{1},\cdots,h_{n}) =$ D$_{m}(k_{1},\cdots,k_{n})$.
\end{theorem}

\begin{proof}
	First, we shall show that D$_{m}(h_{1},\cdots,h_{n}) = q$ implies D$_{m}(k_{1},\cdots,k_{n}) \leq q$. Let $A = X_{0} \cup \cdots \cup X_{q}$, and for each $X_{j}$, $j = 0,\cdots,q$, any $(\delta,\kappa_{1})-$continuous map $\phi : P \rightarrow X_{j}$ admits the property that $h_{l} \circ \phi \simeq_{\delta,\lambda_{1}} h_{l+1} \circ \phi$ for each $l \in \{1,\cdots,n-1\}$. Define $Y_{i} = \omega_{2}^{-1}(X_{i}) \subset A^{'}$ for each $i$. Then $A^{'}$ can be written as the union of $\{Y_{0},\cdots,Y_{q}\}$. Moreover, for each $Y_{j}$, $j = 0,\cdots,q$, any $(\delta,\kappa^{'}_{1})-$continuous map $\phi^{'} : P \rightarrow Y_{j}$ admits the property
	\begin{eqnarray*}
		k_{l} \circ \phi^{'} &\simeq_{\delta,\lambda^{'}_{1}}& k_{l} \circ \omega_{2}^{-1} \circ \phi  \simeq_{\delta,\lambda^{'}_{1}} \omega_{1} \circ h_{l} \circ \omega_{2} \circ \omega_{2}^{-1} \circ \phi \simeq_{\delta,\lambda^{'}_{1}} \omega_{1} \circ h_{l} \circ \phi\\ 
		&\simeq_{\delta,\lambda^{'}_{1}}& \omega_{1} \circ h_{l+1} \circ \phi \simeq_{\delta,\lambda^{'}_{1}} \omega_{1} \circ h_{l+1} \circ \omega_{2} \circ \omega_{2}^{-1} \circ \phi \simeq_{\delta,\lambda^{'}_{1}} k_{l+1} \circ \omega_{2}^{-1} \circ \phi\\ &\simeq_{\delta,\lambda^{'}_{1}}& k_{l+1} \circ \phi^{'}
	\end{eqnarray*}
    by considering the equality $\phi^{'} = \omega_{2}^{-1} \circ \phi$.
	Therefore, D$_{m}(k_{1},\cdots,k_{n}) \leq q$. In addition, the same way can be used to show that D$_{m}(k_{1},\cdots,k_{n}) = q$ implies D$_{m}(h_{1},\cdots,h_{n}) \leq q$ by using the equality $\phi = \omega_{2} \circ \phi^{'}$.
\end{proof}

\begin{corollary}
	Each of cat$_{m}(A,\kappa_{1})$, cat$_{m}(h;\kappa_{1},\lambda_{1})$, TC$^{m}(A,\kappa_{1})$, TC$^{m}(h;\kappa_{1},\lambda_{1})$, and $n-$TC$^{m}(A,\kappa_{1})$ is digitally homotopy invariant.
\end{corollary}

\quad We define a digital fiber $m-$homotopy equivalence between two functions as follows: Let $h : (A,\kappa_{1}) \rightarrow (A^{'},\lambda_{1})$ and $k : (B,\kappa^{'}_{1}) \rightarrow (A^{'},\lambda_{1})$ be two digitally continuous functions. Then they are digitally fiber $m-$homotopy equivalent provided that the following diagram commutes with the properties $v \circ w \circ \phi \simeq_{\delta,\kappa^{'}_{1}} \phi$ and $w \circ v \circ \phi^{'} \simeq_{\delta^{'},\kappa_{1}} \phi^{'}$ for any digitally continuous maps $\phi : (P,\delta) \rightarrow (B,\kappa^{'}_{1})$ and $\phi^{'} : (P^{'},\delta^{'}) \rightarrow (A,\kappa_{1})$, where $P$ and $P^{'}$ are any digital $m-$dimensional complexes. 
$$\xymatrix{
	A \ar[dr]_{h} \ar@<1ex>[rr]^v
	& & B \ar@<1ex>[ll]^w \ar[dl]^{k} \\
	& A^{'}. & }$$

Note that the digital fiber homotopy equivalence implies that the digital fiber $m-$homotopy equivalence. Indeed, $v \circ w \circ \phi \simeq_{\delta,\kappa^{'}_{1}} \phi$ and $w \circ v \circ \phi^{'} \simeq_{\delta^{'},\kappa_{1}} \phi^{'}$ are weaker conditions than $v \circ w \simeq_{\kappa^{'}_{1},\kappa^{'}_{1}} 1_{B}$ and $w \circ v \simeq_{\kappa_{1},\kappa_{1}} 1_{A}$, respectively. Thus, we get the following consequence by considering the digital setting of [Corollary 3.14,\cite{MelihKaraca:2022}]:

\begin{proposition}
	$n-$TC$^{m}(A,\kappa_{1})$ is digitally fiber $m-$homotopy invariant.
\end{proposition}

\section{Future Works}
\label{conc:5}
 
\quad Just as the concept of homotopy has important applications as homotopic distance, the concept of $m-$homotopy also has important applications in the digital images as digital $m-$homotopy. However,  there are still several open problems that need more investigation and study. One of the main difficulties is to extend $m$-homotopic distance to more complicated structures, such as higher-dimensional digital manifolds and networks, and to build efficient algorithms for computing these distances in large-scale, high-dimensional digital environments. Moreover, including $m$-homotopic distance in machine learning frameworks raises concerns about how best to use topological data to improve model interpretability and performance. The theoretical foundations of $m$-homotopic distance in digital topology also require further investigation to understand its properties and relationships with other topological invariants. Exploring these works can greatly enhance the practical applications of $m$-homotopic distance in numerous engineering and scientific fields.

\end{document}